\documentclass[conf]{new-aiaa}
\usepackage[utf8]{inputenc}

\usepackage{graphicx}
\usepackage{amsmath}

\usepackage{Theorems}
\usepackage[version=4]{mhchem}
\usepackage{siunitx}
\usepackage{longtable,tabularx}
\usepackage{color}

\DeclareMathOperator{\vspan}{span}

\title{An occupation kernel approach to optimal control}

\author{Rushikesh Kamalapurkar\footnote{Assistant Professor, School of Mechanical Engineering, Oklahoma State University, Stillwater, OK 74078.} and Joel A. Rosenfeld\footnote{Assistant Professor, Department of Mathematics and Statistics, University of South Florida, Tampa, Fl 33620.}}

\begin{document}

\maketitle

\begin{abstract}
In this effort, a novel operator theoretic framework is developed for data-driven solution of optimal control problems. The developed methods focus on the use of trajectories (i.e., time-series) as the fundamental unit of data for the resolution of optimal control problems in dynamical systems. Trajectory information in the dynamical systems is embedded in a reproducing kernel Hilbert space (RKHS) through what are called occupation kernels. The occupation kernels are tied to the dynamics of the system through the densely defined Liouville operator. The pairing of Liouville operators and occupation kernels allows for lifting of nonlinear finite-dimensional optimal control problems into the space of infinite-dimensional linear programs over RKHSs.
\end{abstract}



\section{Introduction}
\lettrine{N}{umerical}  solutions of optimal control problems are obtained by using Pontryagin's maximum principle \cite{SCC.Pontryagin.Boltyanskii.ea1962} to convert the optimal control problem into a two-point boundary value problem \cite{SCC.Stryk.Bulirsch1992,SCC.Betts1998} or a nonlinear programming problem \cite{SCC.Hargraves.Paris1987,SCC.Huntington2007,SCC.Fahroo.Ross2008,SCC.Rao.Benson.ea2010,SCC.Darby.Hager.ea2011,SCC.Garg.Hager.ea2011}. While there is a rich history of literature on the topic of numerical optimal control, the computational efficiency of numerical optimal control is limited by that of nonlinear programming, where solutions of large problems can be computationally prohibitive and the solutions, when available, are typically only locally optimal.

Based on the seminal work of Lasserre \cite{SCC.Lasserre2010} on moments and positive polynomials, occupation measure approaches that convert a nonlinear optimal control problem into an infinite dimensional linear program that can be efficiently solved using sum of squares based convex programming methods were developed in results such as \cite{SCC.Lasserre.Henrion.ea2008,SCC.Majumdar.Vasudevan.ea2014,SCC.Claeys.Daafouz.ea2016,SCC.Zhao.Mohan.ea2017}.

While computationally efficient, techniques that utilize occupation measures are typically only applicable to systems where the functions that describe the dynamics, the cost functions, and the constraint sets are polynomials. The techniques developed in this paper also convert finite dimensional nonlinear optimal control problems into infinite dimensional linear programs, but utilize a reproducing kernel Hilbert space framework. An advantage of framing the infinite dimensional linear program within the reproducing kernel Hilbert space framework is that the developed tools are applicable to optimal control problems with a broader range of cost functions and constraining sets. Principally, the advantage is realized by exchanging the moment problem for occupation measures with the more flexible approximation abilities of reproducing kernel Hilbert spaces.

\section{Reproducing Kernel Hilbert Spaces}

\begin{definition}
    A real-valued \emph{reproducing kernel Hilbert space} (RKHS), $H$, over a set $X \subset \mathbb{R}^n$ is a Hilbert space of functions $f: X \mapsto \mathbb{R}$ such that for every $x\in X$, the evaluation functional $E_x f := f(x)$ is bounded.
\end{definition}
By the Reisz representation theorem, for each $x \in X$ there is a corresponding function $ k_x \in H $ such that $ \langle f, k_x \rangle_H = f(x) $, where $\langle f,g \rangle_H$ denotes the inner product. For each RKHS, there is a uniquely identified \emph{kernel function}, $K(x,y) := \langle k_y, k_x \rangle_H$, such that for any finite collection of points, $\{ x_i \}_{i=1}^M$, the corresponding Gram matrix, $(K(x_i,x_j))_{i,j=1}^M$, is positive semi-definite.

The importance of RKHSs lies in their ability to perform as function approximators. In particular, just as the collection of polynomials is dense inside of the space of continuous functions over compact subsets of $\mathbb{R}^n$, \emph{universal} RKHSs are those spaces that are also dense in the space of continuous functions over compact subsets of $\mathbb{R}^n$. Moreover, the following lemma demonstrates that it is sufficient to consider linear combinations of the kernel functions themselves for function approximation when the kernel is in a universal RKHS.
\begin{lemma}
    Consider the subset $S := \{ K(\cdot,y) : y \in X \}$ of a RKHS $H$ over a set $X$ with kernel $K$. Then $\vspan S$ is dense in $H$ with respect to the Hilbert space norm. Moreover, if $K$ is continuous, then $\vspan S$ is dense in $H$ with respect to the uniform norm over restrictions to compact subsets of $X$.
\end{lemma}
\begin{proof}
    See \cite[Theorem 4.21]{SCC.Steinwart.Christmann2008}.
\end{proof}
\section{Problem formulation}
Let $H(Y)$ be a real-valued RKHS of continuous functions over the set $Y$. Let $X$ and $D$ be compact subsets of $\mathbb{R}^n$, $U$ a compact subset of $\mathbb{R}^m$, $\Sigma := [0,T] \times X$, and $S = \Sigma \times U$. Throughout the rest of this manuscript the RKHSs $H\left(X\right)$, $H\left(D\right)$ and $H\left(\Sigma\right)$ denote the RKHSs obtained through the functions in $H\left(S\right)$ where the inputs have been projected to $X$, $D$, and $\Sigma$, respectively. Let $f:[0,T] \times \mathbb{R}^n \times \mathbb{R}^m \to \mathbb{R}^n$ be a locally Lipschitz function and consider the dynamical system 
\begin{equation}\label{eq:dynamics}
    \dot x = f(t,x,u), \qquad x(0) = x_0 \in \mathbb{R}^n.
\end{equation}
A state of the dynamical system corresponding to the initial condition $x_0$ and controller $u:[0,T] \mapsto \mathbb{R}^m$ will be written as $\phi_f(t;x_0,u)$.

For a fixed $T$, the optimal control problem is formulated as the need to minimize the cost
\begin{equation}\label{eq:costFunctional}
    J(x(\cdot),u(\cdot)) = \int_0^T h(t,x(t),u(t)) \mathrm{d}t + F(x(T)),
\end{equation}
for functions $h \in H(S)$ and $F \in H(D)$, over the set of differentiable functions $x:[0,T]\to\mathbb{R}^n$ and continuous functions $u[0,T]\to \mathbb{R}^m$ subject to the constraints \eqref{eq:dynamics}. For ease of exposition, the formulation considered here is more restrictive than strictly necessary. The methods developed in the following can be extended to include measurable control signals and absolutely continuous state trajectories.

In the following, occupation kernels and Liouville operators, first introduced in \cite{arXivSCC.Rosenfeld.Russo.ea2020} are utilized to lift the nonlinear optimal control problem into the space of infinite-dimensional linear programs.

\section{Occupation kernels and the cost functional}
Whenever $ (t,x(t),u(t)) \in S $ for all $ t \in [0,T] $, the functional $ g \mapsto \int_0^T g(t,x(t),u(t)) \mathrm{d}t $, that maps from $H(S)$ to $\mathbb{R}$, is linear and bounded. Indeed, given the kernel function $ K_S $ corresponding to $H(S)$, it can be seen that
\[
    \left| \int_0^T g(t,x(t),u(t)) \mathrm{d}t \right| \le \left| \int_0^T \langle g, K_S(\cdot,(t,x(t),u(t))) \rangle_{H(S)} \mathrm{d}t \right|\]
    \[\le  T \| g \|_{H(S)} \sup_{[0,T]} \sqrt{K_S((t,x(t),u(t)),(t,x(t),u(t)))} \le T \| g \|_{H(S)} \sup_{y \in S} \sqrt{K_S(y,y)}.
\]
As such, by the Reisz representation theorem, there exists a function $\Gamma_{x(\cdot),u(\cdot)} \in H(S)$ such that $\int_0^T g(t,x(t),u(t)) \mathrm{d}t = \langle g, \Gamma_{x(\cdot),u(\cdot)}\rangle_{H(S)}$. The function $\Gamma_{x(\cdot),u(\cdot)}$ is the \emph{occupation kernel} corresponding to the signals $x(\cdot)$ and $u(\cdot)$. Note that at this juncture, the signals $x(\cdot)$ and $u(\cdot)$ are independent, i.e., $x(\cdot)$ is not necessarily a trajectory of the dynamical system \eqref{eq:dynamics} in response to $u(\cdot)$.

The occupation kernel itself may be expressed as $\Gamma_{x(\cdot),u(\cdot)}(y) = \int_0^T K_S(y,(t,x(t),u(t))) \mathrm{d}t$. Moreover,
\begin{equation}\label{eq:occNorm}
    \| \Gamma_{x(\cdot),u(\cdot)}\|^2_{H(S)} = \int_0^T\int_0^T K((\tau,x(\tau),u(\tau)),(t,x(t),u(t)))d\tau \mathrm{d}t,
\end{equation}
and when $K(x,y) = \Phi(\|x-y\|_2)$ is a radial basis function, such as the Wendland RBF or the Gaussian RBF, \eqref{eq:occNorm} may be bounded as $\| \Gamma_{x_0,u}\|^2_{H(S)} \le T^2 \Phi(0)$

Using the occupation kernels and the reproducing property $\langle F,K_D(\cdot,y)\rangle H_D = F(y)$ of the kernel function $K_D \in H_D$ corresponding to the RKHS $H(D)$, the cost functional in \eqref{eq:costFunctional} can be expressed as
\begin{equation}\label{eq:occKerCostFunctional}
    J(x(\cdot),u(\cdot)) = \langle h, \Gamma_{x(\cdot),u(\cdot)} \rangle_{H(S)} + \langle F, K_D(\cdot, x(T)) \rangle_{H(D)}.
\end{equation}

Note that the cost functional is linear with respect to the kernels $\Gamma_{x(\cdot),u(\cdot)}$ and $K_D$. If the dynamical system that constrains $x(\cdot)$ to be a solution in response to $u(\cdot)$ can also be expressed as a linear constraint on the space of kernels, the optimal control problem can be posed as a linear program in the infinite dimensional kernel space. 

\section{System dynamics and the total derivative operator}
In the following, a formulation of the dynamics in terms of total derivative operators is developed to construct the aforementioned linear constraint.
\begin{definition}\label{def:totalderivative}
    Define the \emph{total derivative operator with symbol $f$} denoted by $A_f: \mathcal{D}(A_f) \to H(S)$ as $A_f g(t,x) := \frac{\partial}{\partial t} g(t,x) + f(t,x,u) \cdot \nabla_x g(t,x)$ where the domain $\mathcal{D}(A_f)$ is defined cannonically as
    \begin{equation}\label{eq:totalderivative}
        \mathcal{D}(A_f) = \{ g \in H(\Sigma) : A_f g \in H(S) \}.
    \end{equation}
\end{definition}
The total derivative operator is seldom a compact operator. As such, to analyze the relationship between the total derivative operator and the occupation kernels, the theory of densely defined operators is leveraged.
\begin{definition}[Densely Defined Operator]
    Given a set $\mathcal{D}(A)\subset H$, a linear operator $A : \mathcal{D}(A) \to H$ is said to be densely defined when $\mathcal{D}(A)$ is dense in $H$.
\end{definition}

Differentiation is a canonical example of a densely defined operator. The following example, while not posed over a RKHS, demonstrates this property of differentiation over the Hilbert space $L^2[0,1]$.

\begin{example}
    Let $A = \frac{d}{dt}$ and suppose that the Hilbert space in question is $L^2[0,1]$. Since the derivative of any polynomial is again a polynomial and polynomials are dense in $L^2[0,1]$, $\mathcal{D}(A) := \{ p : p \text{ is a polynomial over } [0,1]\}$ is a dense domain for $A$. It is also clear that $\mathcal{D}(A)$ cannot be extended to all of $L^2[0,1]$ as $f(t) = \sqrt{t}$ is in $L^2[0,1]$ and $\frac{d}{dt} f(t) = \frac{1}{2\sqrt{t}}$ is not.
\end{example}
The relationship between the total derivative operator and the occupation kernels is expressed through the adjoint of the total derivative operator, and for the development to be cogent, the adjoint needs to be densely defined. Since adjoints of closed operators over a Hilbert space are densely defined \cite[Chapter 5]{pedersen2012analysis}, closedness of the total derivative operator is analyzed in the following.
\begin{definition}
    Let $A$ be an operator over $H$. $A$ is said to be closed, if whenever $\{ g_m \}_{m=1}^\infty \subset A$, $g_n \to f$ and $Ag_n \to h$ according to the Hilbert space norm, then $f \in \mathcal{D}(A)$ and $Af = h$.
\end{definition}

The following theorem establishes a connection between the total derivative operator and signals $x(\cdot)$ and $u(\cdot)$ whenever $x(\cdot)$ is a solution of \eqref{eq:dynamics} under $u(\cdot)$. For brevity of notation, let $\Gamma_{x_0,u}$ denote the occupation kernel $\Gamma_{\phi_f(\cdot,x_0,u(\cdot)),u(\cdot)}$.

\begin{theorem}\label{thm:adjointdomain}
    The operator $A_f$ introduced in Definition \ref{def:totalderivative} is closed. Moreover, for an admissible trajectory $t \mapsto (t,x(t),u(t))$, with initial condition $x_0$, and that resides within a compact set for all $t \in [0,T]$, the function $\Gamma_{x_0,u}$ is in the domain of the adjoint of $A_f$.
\end{theorem}
\begin{proof}
    Suppose that $\{ g_m\}_{m=0}^\infty \subset \mathcal{D}(A_f) \subset H(\Sigma)$ such that $g_m \to g \in H(\Sigma)$ and $A_f g_m \to q \in H(S)$. Since the differentiability of the functions in $H$ is inherited from the kernel function (see \cite[Corollary 4.36]{SCC.Steinwart.Christmann2008}), the function $\frac{\partial}{\partial x_i} g$ is well defined for each $g\in H(\Sigma)$ (but $\frac{\partial}{\partial x_i} g$ is not necessarily a function in $H(\Sigma)$). However, for any fixed $t$ and $x$ the mapping $p \mapsto \frac{\partial}{\partial x_i} p(t,x)$ is a continuous linear functional over $H(\Sigma)$. By \cite[Corollary 4.36]{SCC.Steinwart.Christmann2008},
    \begin{gather*}
        \left| \frac{\partial}{\partial x_i} g_m(t,x) - \frac{\partial}{\partial x_i} g(t,x) \right| = \left| \frac{\partial}{\partial x_i}\left( g_m(t,x) -  g(t,x)\right)\right| \le \|g_m - g\|_{H(\Sigma)} \sqrt{\partial_i \partial_{i+n} K_\Sigma(((t,x)),((t,x)))}.
    \end{gather*}
    Hence, $\frac{\partial}{\partial x_i} g_m(t,x) \to \frac{\partial}{\partial x_i} g(t,x)$ for each $x \in X$ and $i=1,\ldots,n$. Hence, $\frac{\partial}{\partial t} g_m(t,x) + f(t,x,u) \cdot \nabla_x g_m(t,x) \to \frac{\partial}{\partial t} g(t,x) + f(t,x,u) \cdot \nabla_x g(t,x)$ as $f(t,x,u)$ is constant with respect to $m$. Thus, $h = Ag$ and $g \in \mathcal{D}(A_f)$, and $A_f$ is closed with the domain given in Definition \ref{def:totalderivative}.

    To demonstrate that $\Gamma_{x_0,u}$ is in the domain of $A_f^*$, note that
    \begin{gather*}
        \left| \int_0^T \frac{\partial}{\partial t}g(t,x(t)) + f(t,x(t),u(t)) \nabla_x g(t,x(t)) dt \right| = \left| \int_0^T \dot g(t,x(t)) dt \right| = \left| g(T,x(T)) - g(0,x(0)) \right|\\ = \left|\langle g, K_\Sigma(\cdot,(T,x(T))) - K_\Sigma(\cdot,(0,x(0)))\rangle_{H(\Sigma)} \right| \le \| g\|_{H(\Sigma)} \| K_\Sigma(\cdot,(T,x(T))) - K_\Sigma(\cdot,(0,x(0)))\|_{H(\Sigma)}.
    \end{gather*}

    Finally, given bounds on $T$ and $\|x(t)\|_2$, a bound on $\| K_\Sigma(\cdot,(T,x(T))) - K_\Sigma(\cdot,(0,x(0)))\|_{H(\Sigma)}$ may be established. Thus, the functional over $\mathcal{D}(A_f)$ given as $g \mapsto \langle A_fg, \Gamma_{x_0,u} \rangle$ is bounded when $t\mapsto (t,x(t),u(t))$ is a trajectory of the system.  It follows that the function $\Gamma_{x_0,u}$ is in the domain of the adjoint of the operator $A_f$. That is,
    \begin{equation}\label{eq:adjointinnerproduct}
        \langle A_f g, \Gamma_{x_0,u} \rangle_{H(S)} = \langle g, A_f^* \Gamma_{x_0,u} \rangle_{H(\Sigma)} = g(T,x(T)) - g(0,x(0))
    \end{equation}
    for all $ g \in \mathcal{D}(A_f)$.
\end{proof}
Through consideration of \eqref{eq:adjointinnerproduct} for an admissible trajectory satisfying the hypothesis of Theorem \ref{thm:adjointdomain}, $g \in \mathcal{D}(A_f)$ and setting $g_T(x) \equiv g(T,x) \in H(D)$, it can be observed that
\begin{gather*}\langle g, K_\Sigma(\cdot, (0,x_0) ) \rangle_{H(\Sigma)} = g(0,x(0)) = - \langle g, A_f^* \Gamma_{x_0,u} \rangle + g(T,x(T)) = \langle -A_fg,  \Gamma_{x_0,u} \rangle_{H(S)} + \langle g_T, K_D(\cdot,x(T))\rangle_{H(D)} \\
= \langle (-A_fg,g_T), (\Gamma_{x_0,u}, K_D(\cdot,x(T)))\rangle_{H(S) \times H(D)}.
\end{gather*}
Letting $\mathcal{L}_f: \mathcal{D}(A_f)  \to H(S) \times H(D)$ denote the linear mapping $\mathcal{L}_f g = \left( -A_f g, g_T \right)$, it follows that $ \langle g, \mathcal{L}_f^*( \Gamma_{x_0,u}, K_D(\cdot,x(T))) \rangle_{H(\Sigma)} = \langle g, K_\Sigma(\cdot, (0,x_0) ) \rangle_{H(\Sigma)} $ for all $g \in H(\Sigma)$. Hence, the system dynamics are encoded by the linear constraint
\begin{equation}\label{eq:adjointofL}
    \mathcal{L}_f^*( \Gamma_{x_0,u}, K_D(\cdot,x(T)) ) = K_\Sigma(\cdot, (0,x_0) ).
\end{equation}
\section{A reformulation of the optimal control problem}
Using \eqref{eq:occKerCostFunctional} and \eqref{eq:adjointofL}, the optimal control problem is expressed as an infinite dimensional linear program 
\begin{gather*}
    P : \min_{\Gamma_{x_0,u},K_D(\cdot,x(T))} \langle (\Gamma_{x_0,u},K_D(\cdot,x(T))), (h,F) \rangle_{H(S)\times H(D)}\\
    \text{subject to: } \mathcal{L}^*_{f}(\Gamma_{x_0,u},K_D(\cdot,x(T))) = K_\Sigma(\cdot,(0,x_0)).
\end{gather*}

To solve $P$, finite-dimensional representation of the decision variables $\Gamma_{x_0,u}$ and $K_D(\cdot,x(T))$ is required. The representation is cogent under the following assumptions.
\begin{assumption}\label{ass:denselyDefined}
    $A_f$ is densely defined on $H(\Sigma)$ together with a countable basis for $\mathcal{D}(A_f)$, given as $\{ \sigma_{m} \}_{m=1}^\infty \subset \mathcal{D}(A_f)$. Furthermore, for all $s\in S$, the kernel functions satisfy $K_S(\cdot,s)\in \mathcal{D}(A_f)$.
\end{assumption}
Under Assumption \ref{ass:denselyDefined}, the optimal control problem can be expressed as the need to find the optimal real valued weights $\{ w_i \}_{i=1}^{M_S}$ and $\{ v_i \}_{i=1}^{M_D}$ that provide approximations for $\Gamma_{x_0,u}$ and $K_D(\cdot,x(T))$ as
\begin{gather}
    \Gamma_{x_0,u}(\cdot) \approx \sum_{i=1}^{M_S} w_i K_S(\cdot,s_i)\\
    K_D(\cdot,x(T)) \approx \sum_{i=1}^{M_D} v_i K_D( \cdot, d_i),
\end{gather}
where $\{ s_i \}_{i=1}^{M_S} \subset S$ is a collection of center in $S$, and $\{ d_i \}_{i=1}^{M_D} \subset D$ is a collection of centers in $D$. The objective function of $P$ can then be evaluated as
\begin{multline}
    \langle (\Gamma_{x_0,u},K_D(\cdot,x(T))), (h,F) \rangle_{H(S)\times H(D)} \approx  \left\langle \left( \sum_{i=1}^{M_S} w_i K_S(\cdot,s_i),\sum_{i=1}^{M_D} v_i K_D( \cdot, d_i)\right), (h,F) \right\rangle_{ H(S)\times H(D)}\\ = \sum_{i=1}^{M_S} w_i \left\langle K_S(\cdot,s_i),h\right\rangle_{H(s)} + \sum_{i=1}^{M_D} v_i \left\langle K_D( \cdot, d_i),F \right\rangle_{H(D)} = \sum_{i=1}^{M_S} w_i h(s_i) + \sum_{i=1}^{M_D} v_i F(d_i).
\end{multline}

Similarly, the constraint in $P$ is satisfied provided $\langle \mathcal{L}_f g , (\Gamma_{x_0,u},K_D(\cdot,x(T))) \rangle_{ H(S)\times H(D)} = g(0,x_0)$ for all $g\in \mathcal{D}(A_f)$, which in turn, is satisfied provided $\langle \mathcal{L}_f \sigma_m , (\Gamma_{x_0,u},K_D(\cdot,x(T))) \rangle_{ H(S)\times H(D)} = \sigma_m(0,x_0)$ for all $m=1,\cdots,\infty$. Selecting a finite set of basis functions $ \{ \sigma_{1}, \ldots, \sigma_{M_{b}}\} $, the constraint of $P$ can thus be approximated using $M_b$ linear constraints of the form
\begin{equation}
    \sum_{i=1}^{M_D} v_i \sigma_m(T,d_i)  - \sum_{i=1}^{M_S} w_i A_f\sigma_m(s_i) = \sigma_m(0,x_0), \quad m=1,\cdots, M_b.
\end{equation}
The optimal control problem thus admits the finite-rank representation
\begin{gather*}
    P_f : \min_{\{ w_i \}_{i=1}^{M_S},\{ v_i \}_{i=1}^{M_D}} \sum_{i=1}^{M_S} w_i h(s_i) + \sum_{i=1}^{M_D} v_i F(d_i)\\
    \text{subject to: } \sum_{i=1}^{M_D} v_i \sigma_m(T,d_i)  - \sum_{i=1}^{M_S} w_i A_f\sigma_m(s_i) = \sigma_m(0,x_0),\qquad m=1,\cdots, M_b.
\end{gather*}
To ensure that the optimization problem is bounded, \eqref{eq:occNorm} may be employed as $\|\Gamma_{x_0,u}\|^2 \le T^2 \Phi(0),$ when $K_S$ is the Gaussian or Wendland RBF, and $\|K(\cdot,x(T))\|^2 \le \sup_{y \in D} K(y,y)$. Alternatively, $\Phi(0)$ may be replaced by an appropriate supremum bound. Depending on the selection of the kernel, a theoretically achievable approximation of $\Gamma_{x_0,u}$ and $K_D(\cdot,x(T))$ can be justified based on the density (or \emph{fill distance}) of the centers within their respective parent sets.

\section{Conclusion}
In this abstract, the concepts of occupation kernels and total derivative operators are utilized to lift a nonlinear optimal control problem into a linear infinite-dimensional optimal control problem over functions in a RKHS. A finite-rank representation of the infinite-dimensional problem is obtained using kernel functions of the RKHSs and a countable basis for the domain of the total derivative operator. The authors plan to include an expanded introduction that places this work in the context of other lifting techniques such as occupation measures, provide a procedure to extract the optimal value function from a solution of $P_f$, and add a few example problems that demonstrate the utility of the developed methods.

\section{Acknowledgements}
This research was supported by the Air Force Office of Scientific Research (AFOSR) under contract numbers FA9550-20-1-0127 and FA9550-21-1-0134, and the National Science Foundation (NSF) under award 2027976. Any opinions, findings and conclusions or recommendations expressed in this material are those of the author(s) and do not necessarily reflect the views of the sponsoring agencies.

\bibliography{sample}

\begin{thebibliography}{17}
\newcommand{\enquote}[1]{``#1''}
\providecommand{\natexlab}[1]{#1}
\providecommand{\url}[1]{\texttt{#1}}
\providecommand{\urlprefix}{URL }
\expandafter\ifx\csname urlstyle\endcsname\relax
  \providecommand{\doi}[1]{\discretionary{}{}{}https://doi.org/#1}\else
  \providecommand{\doi}[1]{\discretionary{}{}{}\urlstyle{rm}\url{https://doi.org/#1}}\fi

\bibitem[{Pontryagin et~al.(1962)Pontryagin, Boltyanskii, Gamkrelidze, and
  Mishchenko}]{SCC.Pontryagin.Boltyanskii.ea1962}
Pontryagin, L.~S., Boltyanskii, V.~G., Gamkrelidze, R.~V., and Mishchenko,
  E.~F., \emph{The mathematical theory of optimal processes}, Interscience, New
  York, 1962.

\bibitem[{von Stryk and Bulirsch(1992)}]{SCC.Stryk.Bulirsch1992}
von Stryk, O., and Bulirsch, R., \enquote{Direct and indirect methods for
  trajectory optimization,} \emph{Ann. Oper. Res.}, Vol.~37, No.~1, 1992, pp.
  357--373.

\bibitem[{Betts(1998)}]{SCC.Betts1998}
Betts, J.~T., \enquote{Survey of numerical methods for trajectory
  optimization,} \emph{J. Guid. Control Dynam.}, Vol.~21, No.~2, 1998, pp.
  193--207.

\bibitem[{Hargraves and Paris(1987)}]{SCC.Hargraves.Paris1987}
Hargraves, C.~R., and Paris, S.~W., \enquote{Direct trajectory optimization
  using nonlinear programming and collocation,} \emph{J. Guid. Control Dynam.},
  Vol.~10, No.~4, 1987, pp. 338--342.

\bibitem[{Huntington(2007)}]{SCC.Huntington2007}
Huntington, G.~T., \enquote{Advancement and analysis of a {G}auss
  pseudospectral transcription for optimal control,} Ph.D. thesis, Department
  of Aeronautics and Astronautics, MIT, May 2007.

\bibitem[{Fahroo and Ross(2008)}]{SCC.Fahroo.Ross2008}
Fahroo, F., and Ross, I.~M., \enquote{Pseudospectral methods for
  infinite-horizon nonlinear optimal control problems,} \emph{J. Guid. Control
  Dynam.}, Vol.~31, No.~4, 2008, pp. 927--936.

\bibitem[{Rao et~al.(2010)Rao, Benson, Darby, Patterson, Francolin, and
  Huntington}]{SCC.Rao.Benson.ea2010}
Rao, A.~V., Benson, D.~A., Darby, C.~L., Patterson, M.~A., Francolin, C., and
  Huntington, G.~T., \enquote{Algorithm 902: {GPOPS}, a {MATLAB} software for
  solving multiple-phase optimal control problems using the {G}auss
  pseudospectral method,} \emph{ACM Trans. Math. Softw.}, Vol.~37, No.~2, 2010,
  pp. 1--39.

\bibitem[{Darby et~al.(2011)Darby, Hager, and Rao}]{SCC.Darby.Hager.ea2011}
Darby, C.~L., Hager, W.~W., and Rao, A.~V., \enquote{An hp-adaptive
  pseudospectral method for solving optimal control problems,} \emph{Optim.
  Control Appl. Methods}, Vol.~32, No.~4, 2011, pp. 476--502.
\newblock \doi{10.1002/oca.957}.

\bibitem[{Garg et~al.(2011)Garg, Hager, and Rao}]{SCC.Garg.Hager.ea2011}
Garg, D., Hager, W.~W., and Rao, A.~V., \enquote{Pseudospectral methods for
  solving infinite-horizon optimal control problems,} \emph{Automatica},
  Vol.~47, No.~4, 2011, pp. 829--837.

\bibitem[{Lasserre(2010)}]{SCC.Lasserre2010}
Lasserre, J.~B., \emph{Moments, Positive Polynomials and Their Applications},
  Imperial College Press, 2010.

\bibitem[{Lasserre et~al.(2008)Lasserre, Henrion, Prieur, and
  Tr{\'e}lat}]{SCC.Lasserre.Henrion.ea2008}
Lasserre, J.~B., Henrion, D., Prieur, C., and Tr{\'e}lat, E.,
  \enquote{Nonlinear optimal control via occupation measures and
  {LMI}-relaxations,} \emph{SIAM J. Control Optim.}, Vol.~47, No.~4, 2008, pp.
  1643--1666.

\bibitem[{Majumdar et~al.(2014)Majumdar, Vasudevan, Tobenkin, and
  Tedrake}]{SCC.Majumdar.Vasudevan.ea2014}
Majumdar, A., Vasudevan, R., Tobenkin, M.~M., and Tedrake, R., \enquote{Convex
  optimization of nonlinear feedback controllers via occupation measures,}
  \emph{Int. J. Robot. Res.}, Vol.~33, No.~9, 2014, pp. 1209--1230.

\bibitem[{Claeys et~al.(2016)Claeys, Daafouz, and
  Henrion}]{SCC.Claeys.Daafouz.ea2016}
Claeys, M., Daafouz, J., and Henrion, D., \enquote{Modal occupation measures
  and {LMI} relaxations for nonlinear switched systems control,}
  \emph{Automatica}, Vol.~64, 2016, pp. 143--154.

\bibitem[{Zhao et~al.(2017)Zhao, Mohan, and Vasudevan}]{SCC.Zhao.Mohan.ea2017}
Zhao, P., Mohan, S., and Vasudevan, R., \enquote{Control synthesis for
  nonlinear optimal control via convex relaxations,} \emph{Proc. Am. Control
  Conf.}, 2017, pp. 2654--2661.

\bibitem[{Steinwart and Christmann(2008)}]{SCC.Steinwart.Christmann2008}
Steinwart, I., and Christmann, A., \emph{Support vector machines}, Information
  Science and Statistics, Springer, New York, 2008.

\bibitem[{Rosenfeld et~al.(2020)Rosenfeld, Russo, Kamalapurkar, and
  Johnson}]{arXivSCC.Rosenfeld.Russo.ea2020}
Rosenfeld, J., Russo, B., Kamalapurkar, R., and Johnson, T., \enquote{The
  Occupation Kernel Method for Nonlinear System Identification,}
  {arXiv:1909.11792}, 2020.
\newblock Submitted to {SIAM} Journal on Control and Optimization.

\bibitem[{Pedersen(2012)}]{pedersen2012analysis}
Pedersen, G.~K., \emph{Analysis now}, Vol. 118, Springer Science \& Business
  Media, 2012.

\end{thebibliography}

\end{document}